\documentclass[12pt, a4paper, oneside]{amsart}
\usepackage{amsmath, amsthm, amssymb, geometry, comment}
\usepackage[all]{xy}
\usepackage[dvipdfmx, colorlinks=false]{hyperref}

\newtheorem{thm}{Theorem}[section]
\newtheorem{prop}[thm]{Proposition}
\newtheorem{lem}[thm]{Lemma}
\newtheorem{cor}[thm]{Corollary}

\theoremstyle{definition}
\newtheorem{defn}[thm]{Definition}
\newtheorem{conj}[thm]{Conjecture}
\newtheorem{prob}[thm]{Problem}
\newtheorem{ex}[thm]{Example}

\theoremstyle{remark}
\newtheorem{rem}[thm]{Remark}

\def\N{\mathbb{N}}

\def\R{\mathbb{R}}
\def\C{\mathbb{C}}
\def\P{\mathbb{P}}

\def\B{\mathbb{B}}

\def\cO{\mathcal{O}}

\def\cU{\mathcal{U}}

\def\id{{\rm id}}
\def\pr{{\rm pr}}
\def\T{{\rm T}}
\def\d{{\rm d}}
\def\Ell{{\rm Ell}}

\def\Aut{\operatorname{Aut}}
\def\Bl{\operatorname{Bl}}

\begin{document}

\title[Elliptic characterization and localization of Oka manifolds]{Elliptic characterization and localization\\of Oka manifolds}
\author{Yuta Kusakabe}
\address{Department of Mathematics, Graduate School of Science, Osaka University, Toyonaka, Osaka 560-0043, Japan}
\email{y-kusakabe@cr.math.sci.osaka-u.ac.jp}
\subjclass[2010]{32E10, 32H02, 32E30, 54C35}
\keywords{Stein manifold, Oka manifold, ellipticity, localization principle}

\begin{abstract}
We prove that Gromov's ellipticity condition $\Ell_1$ characterizes Oka manifolds.
This characterization gives another proof of the fact that subellipticity implies the Oka property, and affirmative answers to Gromov's conjectures.
As another application, we establish the localization principle for Oka manifolds, which gives new examples of Oka manifolds.
In the appendix, it is also shown that the Oka property is not a bimeromorphic invariant.
\end{abstract}

\maketitle

\section{Introduction}\label{sec:intro}

In 1989, Gromov's seminal paper \cite{Gro} on the Oka principle initiated modern Oka theory.
Forstneri\v{c}, L\'{a}russon and others developed his theory into the theory of Oka manifolds (see the survey \cite{Forsurvey} and the comprehensive monograph \cite{ForSMHM}).
We first recall the definition of Oka manifolds.
Throughout this paper, all complex manifolds are taken to be second countable and connected.

\begin{defn}\label{defn:oka}
A complex manifold $Y$ is said to be an {\em Oka manifold} if any holomorphic map from an open neighborhood of a compact convex set $K\subset\C^n$ to $Y$ can be uniformly approximated on $K$ by holomorphic maps $\C^n\to Y$.
\end{defn}

In general, it is difficult to verify the condition of Definition \ref{defn:oka} directly.
Gromov's Oka principle gives a sufficient condition, called {\em ellipticity}, for a manifold to be Oka.
A complex manifold $Y$ is said to be {\em elliptic} if there exists a holomorphic map $s:E\to Y$ defined on the total space of a holomorphic vector bundle over $Y$, such that $s(0_y)=y$ and $s|_{E_y}\to Y$ is a submersion at $0_y$ for each $y\in Y$.
Forstneri\v{c} \cite{Forsubell} proved that more general subellipticity and weak subellipticity imply the Oka property (for the definitions, see \cite[Definition 5.6.13]{ForSMHM}).
It is a well-known problem whether the Oka property can be characterized by ellipticity or its variants.

There is another ellipticity condition introduced by Gromov, which is called {\em Condition $\Ell_1$}.
In contrast to ellipticity mentioned above, this condition can be considered as {\it relative ellipticity}.

\begin{defn}\label{defn:ell1}
A complex manifold $Y$ satisfies {\em Condition $\Ell_1$} if for any Stein manifold $X$ and any holomorphic map $f\in\cO(X,Y)$, there exists a holomorphic map $F:X\times\C^N\to Y$ such that $F(\cdot,0)=f$ and $F(x,\cdot):\C^N\to Y$ is a submersion at $0$ for each $x\in X$.
\end{defn}

It is well-known that the Oka property implies Condition $\Ell_1$ (see \cite[Corollary 8.8.7]{ForSMHM} or Proposition \ref{prop:genell1} below).
Our first main result states that the converse is also true.
Thus Condition $\Ell_1$ is equivalent to the Oka property.

\begin{thm}[Elliptic characterization of Oka manifolds]\label{thm:ell1}
A complex manifold is an Oka manifold if and only if it satisfies Condition $\Ell_1$.
\end{thm}

Recently, L\'{a}russon and Truong \cite{LarTru} showed the equivalence between the algebraic versions of $\Ell_1$ and the Oka property (the homotopy Runge property).
An algebraic manifold satisfying these equivalent properties is said to be {\em algebraically Oka} in their paper.
Theorem \ref{thm:ell1} is the holomorphic counterpart of their result.

Theorem \ref{thm:ell1} has various applications.
Since subellipticity implies Condition $\Ell_1$ immediately (cf. \cite[Proposition 8.8.11]{ForSMHM}), the proof of Theorem \ref{thm:ell1} gives another proof of the fact that (weak) subellipticity implies the Oka property (for weak subellipticity, use Theorem \ref{thm:cell1}).
Some other new characterizations and affirmative answers to Gromov's conjectures in \cite[\S 1.4.E$''$]{Gro} are also given in Section \ref{sec:applications}.
As another application, we shall prove the following second main result.
Here, a subset of $Y$ is said to be Zariski open if its complement is a closed complex subvariety.

\begin{thm}[Localization principle for Oka manifolds]\label{thm:localization}
Let $Y$ be a complex manifold.
Assume that each point of $Y$ has a Zariski open Oka neighborhood.
Then $Y$ is an Oka manifold. 
\end{thm}

Using the idea of Gromov \cite[\S 3.5.B]{Gro}, Forstneri\v{c} \cite{Forsubell} proved the localization principle for algebraic subellipticity.
Since algebraic subellipticity is also equivalent to the algebraic Oka property (L\'arusson-Truong \cite{LarTru}), it can be also viewed as the localization principle for algebraically Oka manifolds (see \cite[Remark 2\,(b)]{LarTru}).
Theorem \ref{thm:localization} is the holomorphic counterpart of this localization principle.

Theorem \ref{thm:localization} gives new examples of Oka manifolds.
Some of these examples are given in Section \ref{sec:applications}, but we give here another example.
It was proved by L\'{a}russon that every smooth toric variety is Oka (see \cite[Theorem 2.17]{Forsurvey}).
Such a variety is Zariski locally isomorphic to an algebraic torus $(\C^*)^n$.
In his proof, however, the quotient construction of toric varieties was used instead of the algebraic localization principle because $(\C^*)^n$ is not algebraically Oka.
Now we have the holomorphic localization principle, which gives the direct proof of this fact.
Moreover, since the algebraic torus $(\C^*)^n$ with finitely many points removed, or blown up at finitely many points, is Oka for $n>1$ (see \cite[Corollary 5.6.18 and Corollary 6.4.13]{ForSMHM}), we have the following result.

\begin{cor}\label{cor:toric}
Let $Y$ be a complex manifold of dimension at least two.
Assume that $Y$ is Zariski locally isomorphic to $(\C^*)^n$.
Then for any finite subset $A\subset Y$ the complement $Y\setminus A$ and the blowup $\Bl_AY$ are Oka.
This holds in particular for any smooth toric variety of dimension at least two.
\end{cor}

Related to this, L\'arusson and Truong proved that every smooth nondegenerate toric variety is locally flexible \cite[Theorem 3]{LarTru}.
It is known that for any locally flexible algebraic manifold $Y$ and any algebraic submanifold $A\subset Y$ of codimension at least two, the complement $Y\setminus A$ and the blowup $\Bl_AY$ are algebraically Oka \cite{FleKalZai, KalKutTru}.
It seems that, if $Y$ is toric, Corollary \ref{cor:toric} also holds for any algebraic submanifold $A\subset Y$ of codimension at least two.
In the case when $A\subset Y$ is nonalgebraic, Forstneri\v{c} proved that for $n>1$ and every tame discrete subset $D\subset\C^n$ the complement $\C^n\setminus D$ and the blowup $\Bl_D\C^n$ is Oka (cf. \cite[Proposition 5.6.17 and Proposition 6.4.12]{ForSMHM}).
On the other hand, Rosay and Rudin \cite{RosRud} constructed a discrete subset $D\subset\C^n$ such that the complement $\C^n\setminus D$ is not Oka for each $n>1$.
In the appendix, using their technique, we shall construct a discrete subset $D\subset\C^n$ such that the blowup $\Bl_D\C^n$ is not Oka for each $n>1$ (Example \ref{ex:blowup}).
In particular, it shows that the Oka property is not a bimeromorphic invariant.

For $n>1$, the complement $\C^n\setminus\overline{\B^n}$ of the closed unit ball and the complement $\C^n\setminus\R^k$ of a totally real affine subspace are unions of Fatou-Bieberbach domains, but they are not known to be Oka (cf. \cite[Problem 7.6.1]{ForSMHM}).
Thus, related to Theorem \ref{thm:localization}, it is natural to ask the following.
It will be studied in future work.

\begin{prob}
Let $Y$ be a complex manifold.
Assume that each point of $Y$ has an open Oka neighborhood.
Does it follow that $Y$ is an Oka manifold?
\end{prob}

This paper is organized as follows.
In Section \ref{sec:ell1}, we recall a few facts about holomorphic sprays and prove Theorem \ref{thm:cell1} which implies Theorem \ref{thm:ell1}. 
In Section \ref{sec:localization}, we prove Corollary \ref{cor:relell1} which is a refinement of Theorem \ref{thm:ell1}, and prove Theorem \ref{thm:localization} as an application.
Section \ref{sec:applications} contains applications of the results and the methods in the previous sections, such as new characterizations of Oka manifolds, affirmative answers to Gromov's conjectures and new examples of Oka manifolds.
In the appendix, by using the technique of Rosay and Rudin \cite{RosRud}, we construct an example of a non-Oka blowup of $\C^n$ for each $n>1$ (Example \ref{ex:blowup}).

\section{Proof of Theorem \ref{thm:ell1}}\label{sec:ell1}

We first recall the notion of holomorphic sprays and their dominability.
A map from a subset $A\subset X$ of a Stein manifold $X$ to a complex manifold $Y$ is said to be holomorphic if it is the restriction of a holomorphic map from an open neighborhood of $A$ to $Y$.
We denote by $\cO(A,Y)$ the space of holomorphic maps $A\to Y$.
Throughout this paper, spaces of holomorphic maps are endowed with the compact-open topology.

\begin{defn}\label{defn:spray}
Let $X$ be a Stein manifold, $Y$ be a complex manifold and $A\subset X$ be a subset.
\begin{enumerate}
\item A {\em holomorphic spray} over $f\in\cO(A,Y)$ is a holomorphic map $F:\Omega\times P\to Y$ satisfying $F(\cdot,0)|_A=f$, where $A\subset\Omega\subset X$ and $0\in P\subset\C^N$ are open neighborhoods.
\item A holomorphic spray $F:\Omega\times P\to Y$ is said to be {\em dominating on a subset $B\subset\Omega$} if the partial differential $\partial_t|_{t=0} F(x,t):\T_0\C^N\to\T_{F(x,0)}Y$ is surjective for each $x\in B$.
A holomorphic spray $F:\Omega\times P\to Y$ dominating on the whole of $\Omega$ is simply called a {\em dominating spray}.
\end{enumerate}
\end{defn}

Note that the holomorphic map $F:X\times\C^N\to Y$ in Definition \ref{defn:ell1} is nothing but a dominating spray over $f\in\cO(X,Y)$.

In this section we shall prove the following result, which implies Theorem \ref{thm:ell1} because every compact convex set admits a basis of open Stein neighborhoods.

\begin{thm}\label{thm:cell1}
A complex manifold $Y$ is Oka if and only if for any compact convex set $K\subset\C^n$ and any holomorphic map $f\in\cO(K,Y)$ there exist an open neighborhood $\Omega\subset\C^n$ of $K$ and a dominating spray $\Omega\times\C^N\to Y$ over $f$.
\end{thm}

To prove this theorem, we briefly review a few facts about holomorphic sprays.
The first two lemmas are special cases of the main tools used in the proof of Forstneri\v{c}'s Oka principle \cite[Theorem 5.4.4]{ForSMHM}.

\begin{lem}[cf. {\cite[Lemma 5.10.4]{ForSMHM}}]\label{lem:localspray}
Let $X$ and $Y$ be complex manifolds, and let $K\subset X$ be a compact subset which has a basis of open Stein neighborhoods.
Then for any holomorphic map $f\in\cO(K,Y)$ there exist open neighborhoods $K\subset\Omega\subset X$ and $0\in P\subset\C^N$, and a dominating spray $\Omega\times P\to Y$ over $f$.
\end{lem}

Recall that a {\em Cartan pair} $(A,B)$ in $X$ is a pair of compact subsets of $X$ such that each of $A,B,A\cap B$ and $A\cup B$ has a basis of open Stein neighborhoods in $X$ and the separation condition $\overline{A\setminus B}\cap\overline{B\setminus A}=\emptyset$ is satisfied.

\begin{lem}[{cf. \cite[Remark 5.9.4\,(C)]{ForSMHM}}]\label{lem:gluing}
Let $X$ and $Y$ be complex manifolds, $(A,B)$ be a Cartan pair in $X$, $F_A:\Omega_A\times P\to Y$ and $F_B:\Omega_B\times P\to Y$ be holomorphic sprays over $f_A\in\cO(A,Y)$ and $f_B\in\cO(B,Y)$, respectively, and $\Omega_{A\cap B}\subset\Omega_A\cap\Omega_B$ be an open neighborhood of $A\cap B$.
Assume that $F_A$ is dominating on $A\cap B$, and $F_B:\Omega_B\times P\to Y$ is sufficiently close to $F_A$ on $\Omega_{A\cap B}\times P$.
Then there exists a holomorphic map $f:A\cup B\to Y$ which is close to $f_A$ on $A$.
If in addition $F_B|_{\Omega_0\times P}$ is sufficiently close to a given holomorphic spray $F_0:\Omega_0\times P\to Y$ over $f_0\in\cO(K_0,Y)$, where $K_0\subset B$ is a compact subset and $\Omega_0\subset\Omega_B$ is its open neighborhood, then $f$ can also be chosen close to $f_0$ on $K_0$.
\end{lem}

The following is also well-known, but we give its proof for reader's convenience.
For $j=1,2$, we denote by $\pr_{j}:X_{1}\times X_{2}\to X_{j}$ the projection onto $X_{j}$.

\begin{lem}\label{lem:section}
Let $X$ be a Stein manifold, $Y$ be a complex manifold and $F:X\times\C^N\to Y$ be a dominating spray over $f=F(\cdot,0)\in\cO(X,Y)$.
Then for the fiber-preserving holomorphic map $(\pr_{1},F):X\times\C^N\to X\times Y$ there exists a fiber-preserving holomorphic map $\sigma:U\to X\times\C^N$ from an open neighborhood $U\subset X\times Y$ of the graph $\Gamma_f=\{(x,f(x)):x\in X\}$ such that $(\pr_{1},F)\circ\sigma=\id_U$.
\end{lem}

\begin{proof}
Let $E$ denote the holomorphic vector subbundle of $X\times\C^N$ with fibers
\begin{align*}
E_x=\ker(\partial_t|_{t=0}F(x,t):\C^N=\T_0\C^N\to\T_{f(x)}Y),\ x\in X.
\end{align*}
Since $X$ is Stein, there exists another holomorphic vector subbundle $E'$ of $X\times\C^N$ such that $X\times\C^N=E\oplus E'$ (cf. \cite[Corollary 2.6.6]{ForSMHM}).
Then the map $(\pr_{1},F)|_{E'}:E'\to X\times Y$ restricts to a biholomorphic map from an open neighborhood of the zero section onto an open neighborhood of the graph $\Gamma_f$.
Its inverse holomorphic map $\sigma$ finishes the proof.
\end{proof}

We also use the technique in the previous paper \cite{Kus} where we proved other characterizations of Oka manifolds.
We may reduce the approximation problem in the definition of Oka manifolds (Definition \ref{defn:oka}) to the following lemma (cf. \cite[Proof of Theorem 3.2]{Kus}; see also \cite[Lemma 5.15.4]{ForSMHM}).

\begin{lem}\label{lem:CAP}
A complex manifold $Y$ is Oka if and only if for any compact convex set $K\subset\C^n$ and any affine linear function $\lambda:\C^n\cong\R^{2n}\to\R$, every holomorphic map $K_{0}=\{z\in K:\lambda(z)\leq 0\}\to Y$ can be uniformly approximated on $K_{0}$ by holomorphic maps $K\to Y$.
\end{lem}

\begin{proof}[Proof of Theorem \ref{thm:cell1}]
We only need to prove the ``if'' part.
We shall verify the condition of Lemma \ref{lem:CAP}.
Let $K\subset\C^n$ be a compact convex set and $\lambda:\C^n\to\R$ be an affine linear function.
Set $K_0=\{z\in K:\lambda(z)\leq 0\}$.
We would like to show that $\cO(K,Y)|_{K_0}\subset\cO(K_0,Y)$ is dense, where $\cO(K,Y)|_{K_0}=\{\varphi|_{K_0}:\varphi\in\cO(K,Y)\}$.
Note that $\cO(K_0,Y)$ is connected because $K_0$ is convex and $Y$ is connected.
Therefore it suffices to prove that the closure $\overline{\cO(K,Y)|_{K_0}}$ is open in $\cO(K_0,Y)$.
Take $f\in\overline{\cO(K,Y)|_{K_0}}$ arbitrarily.
By assumption, there exist an open neighborhood $\Omega\subset\C^n$ of $K_0$ and a dominating spray $F:\Omega\times\C^N\to Y$ over $f$.
Shrinking $\Omega\supset K_0$ if necessary, we may assume that $\Omega$ is a bounded convex domain.
This spray defines a fiber-preserving map $(\pr_{1},F):\Omega\times\C^N\to\Omega\times Y$.
Throughout this proof, we identify maps with sections of trivial bundles without further mention.
Then $(\pr_{1},F)$ defines a continuous map $(\pr_{1},F)_*:\cO(K_0,\C^N)\to\cO(K_0,Y)$ by composition.
Note that, by Lemma \ref{lem:section}, there exists a fiber-preserving holomorphic map $\sigma:U\to\Omega\times\C^N$ from an open neighborhood $U$ of the graph $\Gamma_f\subset\Omega\times Y$ such that $(\pr_{1},F)\circ\sigma=\id_U$.
It implies that the image $(\pr_{1},F)_*(\cO(K_0,\C^N))$ contains an open neighborhood $\{\varphi\in\cO(K_0,Y):(\id,\varphi)(K_0)\subset U\}$ of $f$. 
Thus, to prove that $f$ is an interior point of $\overline{\cO(K,Y)|_{K_0}}$, it suffices to show the inclusion $(\pr_{1},F)_*(\cO(K_0,\C^N))\subset\overline{\cO(K,Y)|_{K_0}}$.

Since $f\in\overline{\cO(K,Y)|_{K_0}}$, there exists a holomorphic map $g\in\cO(K,Y)$ such that $(\id,g)(K_0)\subset U$.
Take a small $\varepsilon>0$ such that $\{z\in K:\lambda(z)\leq 2\varepsilon\}\subset\Omega\cap (\id,g)^{-1}(U)$.
Set $A=\{z\in K:\lambda(z)\geq\varepsilon\}$ and $B=\{z\in K:\lambda(z)\leq 2\varepsilon\}$.
Clearly, $(A,B)$ is a Cartan pair in $\C^n$ with $A\cup B=K$.
By Lemma \ref{lem:localspray}, there exist open neighborhoods $A\subset\Omega_A\subset\C^n$ and $0\in P\subset\C^L$, and a dominating spray $G_A:\Omega_A\times P\to Y$ over $g|_A$.
The first projection $\pi:\Omega\times P\to\Omega$ defines the pullbacks $\pi^*(\pr_{1},F):(\Omega\times P)\times\C^N\to(\Omega\times P)\times Y$ and $\pi^*\sigma:(\Omega\times P)\times_\Omega U\to(\Omega\times P)\times\C^N$ of the fiber-preserving maps $(\pr_{1},F)$ and $\sigma$, respectively.
Note that $\pi^*(\pr_{1},F)\circ\pi^*\sigma=\id_{(\Omega\times P)\times_\Omega U}$ still holds.
Since $(\id,g)(A\cap B)\subset U$, shrinking $P\ni 0$ if necessary, we may assume that there exists an open neighborhood $\Omega_{A\cap B}\subset\Omega_A\cap\Omega$ of $A\cap B$ such that $(\id,G_A(\cdot,t))(\Omega_{A\cap B})\subset U$ for each $t\in P$.
Then $G_A|_{\Omega_{A\cap B}\times P}$ can be considered as a section of $(\Omega\times P)\times_\Omega U\to\Omega\times P$.
Thus, as above, we may consider $G_A'=(\pi^*\sigma)_*(G_A|_{\Omega_{A\cap B}\times P})\in\cO(\Omega_{A\cap B}\times P,\C^N)$ by composition.
By definition, it satisfies $(\pi^*(\pr_{1},F))_*(G_A')=G_A|_{\Omega_{A\cap B}\times P}$.

Take arbitrary $\varphi_0\in\cO(K_0,\C^N)$ and assume that $\varphi_0$ is defined on an open neighborhood $\Omega_0\subset\Omega$ of $K_0$.
Consider the spray $\Phi_0=\varphi_0\circ\pi:\Omega_0\times P\to\C^N$ over $\varphi_0$.
Shrinking $\Omega_0\supset K_0$, $\Omega_{A\cap B}\supset A\cap B$ and $P\ni 0$ if necessary, we may assume that they are convex and $\Omega_0\cap\Omega_{A\cap B}=\emptyset$.
Recall that every union of two disjoint compact convex sets is polynomially convex.
Thus, by Oka-Weil approximation theorem, there exists a holomorphic spray $G_B':\Omega\times P\to\C^N$ which is sufficiently close to $G_A'$ on $\Omega_{A\cap B}\times P$ and sufficiently close to $\Phi_0$ on $\Omega_0\times P$.
Then the spray $G_B=(\pi^*(\pr_{1},F))_*(G_B'):\Omega\times P\to Y$ is sufficiently close to $G_A$ on $\Omega_{A\cap B}\times P$ and sufficiently close to the spray $(\pi^*(\pr_{1},F))_*(\Phi_0)$ over $(\pr_{1},F)_*(\varphi_0)$ on $\Omega_0\times P$.
Then by Lemma \ref{lem:gluing}, there exists a holomorphic map from $K=A\cup B$ to $Y$ which is close to $(\pr_{1},F)_*(\varphi_0)$ on $K_0$.
Thus we obtain the desired inclusion $(\pr_{1},F)_*(\cO(K_0,\C^N))\subset\overline{\cO(K,Y)|_{K_0}}$.
\end{proof}

The above proof shows the following approximation theorem which may be of independent interest.

\begin{thm}
Let $K\subset\C^n$ be a compact convex set, $\lambda:\C^n\to\R$ be an affine linear function, $Y$ be a complex manifold and $F:\Omega\times\C^{N}\to Y$ be a dominating spray over $f\in\overline{\cO(K,Y)|_{K_0}}$ where $K_{0}=\{z\in K:\lambda(z)\leq 0\}$.
Then for any holomorphic map $\varphi\in\cO(K_0,\C^N)$ the inclusion $F\circ(\id,\varphi)\in\overline{\cO(K,Y)|_{K_0}}$ holds.
\end{thm}

We end this section with some remarks about Theorem \ref{thm:ell1}.

\begin{rem}
(1) In the previous paper \cite{Kus}, we proved that several holomorphic flexibility properties such as strong $\C$-connectedness characterize Oka manifolds if we generalize these properties to spaces of holomorphic maps.
Since we may consider Condition $\Ell_1$ as strong dominability for spaces of holomorphic maps (cf. \cite[Definition 7.1.7]{ForSMHM}), Theorem \ref{thm:ell1} is a characterization theorem of this sort.
\\
(2) For a Stein manifold $Y$, it is easily seen that ellipticity, the Oka property and Condition $\Ell_1$ are equivalent.
Indeed, if $Y$ satisfies Condition $\Ell_1$ then there exists a dominating spray $s:Y\times\C^{N}\to Y$ over the identity map, hence $Y$ is elliptic.
\end{rem}

\section{Proof of Theorem \ref{thm:localization}}\label{sec:localization}

In the proof of Theorem \ref{thm:localization}, we shall need the following proposition.

\begin{prop}\label{prop:genell1}
Let $Y$ be a complex manifold and $U\subset Y$ be a Zariski open Oka subset.
Then for any Stein manifold $X$ and any holomorphic map $f\in\cO(X,Y)$ there exists a holomorphic spray $X\times\C^N\to Y$ over $f$ dominating on $f^{-1}(U)$. 
\end{prop}

Our proof is based on the proof of the implication from the Oka property to $\Ell_1$ (cf. \cite[Corollary 8.8.7]{ForSMHM}).
To generalize it, we need the following Oka principle which is a special case of Forstneri\v{c}'s Oka principle for elimination of intersections (\cite[Theorem 8.6.1]{ForSMHM} with $\Sigma=\Gamma_{f|_{X'}}\cup(X\times A)\subset X\times Y$).
A slightly weaker version was used by Forstneri\v{c} and L\'{a}russon \cite{ForLar} to prove that every stratified Oka manifold is strongly dominable (for the definitions, see \cite[Definition 7.1.7]{ForSMHM}).

\begin{thm}[{cf. \cite[Theorem 8.6.1]{ForSMHM}}]\label{thm:okaprinciple}
Let $X$ be a Stein manifold, $X'\subset X$ be a closed complex subvariety, $Y$ be a complex manifold, $A\subset Y$ be a closed complex subvariety and $f:X\to Y$ be a continuous map which is holomorphic on an open neighborhood of $X'$.
Assume that $Y\setminus A$ is Oka and $f^{-1}(A)\subset X'$.
Then, for each $k\in\N$, there exists a homotopy of continuous maps $f_t:X\to Y,\ t\in[0,1]$ with $f_0=f$ and $f_1\in\cO(X,Y)$ such that for each $t\in[0,1]$ the map $f_t$ is holomorphic on an open neighborhood of $X'$, it agrees with $f$ to order $k$ along $X'$, and it satisfies $f_t^{-1}(A)=f^{-1}(A)$.
\end{thm}

\begin{proof}[Proof of Proposition \ref{prop:genell1}]
Set $A=Y\setminus U$.
Since the graph $\Gamma_f\subset X\times Y$ of $f$ is Stein, it has an open Stein neighborhood $\Omega\subset X\times Y$ by Siu's theorem \cite{Siu}.
Let $E\to\Omega$ denote the restriction of the vertical tangent bundle of the trivial bundle $X\times Y\to X$.
Since $\Omega$ is Stein, Cartan's theorem A implies that there exist finitely many generators of $E$, namely holomorphic vector fields $V_k,\ k=1,\ldots,N$ on $\Omega$ tangent to the fibers of $X\times Y\to X$.
For the same reason, there exist finitely many holomorphic functions $g_j\in\cO(\Omega),\ j=1,\ldots,L$ such that
\begin{align*}
\{z\in\Omega:g_j(z)=0,\ j=1,\ldots,L\}=\Omega\cap(X\times A).
\end{align*}
Let $\varphi_t^{j,k}$ denote the local flow of $g_jV_k$.
Since each of $g_jV_k$ vanishes on $\Omega\cap(X\times A)$, each flow fixes $\Omega\cap(X\times A)$ and satisfies $(\varphi_t^{j,k})^{-1}(\Omega\cap(X\times A))\subset\Omega\cap(X\times A)$.

Let $\pr_{2}:X\times Y\to Y$ denote the second projection.
Then
\begin{align*}
F_0'(x,(t_{1,1},\ldots,t_{1,N}),\ldots,(t_{L,1},\ldots,t_{L,N}))=\pr_{2}\circ\varphi_{t_{1,1}}^{1,1}\circ\cdots\circ\varphi_{t_{L,N}}^{L,N}(x,f(x))\in Y
\end{align*}
defines a holomorphic map from an open neighborhood of $X\times\{0\}\subset X\times\C^{LN}$.
Consider $F_0(x,t_1,\ldots,t_L)=F_0'(x,g_1(x,f(x))t_1,\ldots,g_L(x,f(x))t_L)$, where we write $t_j=(t_{j,1},\ldots,t_{j,N})$.
By construction, it is a holomorphic map from an open neighborhood of $X'=(X\times\{0\})\cup(f^{-1}(A)\times\C^{LN})$, which satisfies $F_0(\cdot,0)=f$, $F_0^{-1}(A)=f^{-1}(A)\times\C^{LN}$ and $\partial_t|_{t=0}F_0(x,t)$ is surjective for all $x\in f^{-1}(U)$.
Shrinking the open neighborhood if necessary, we can extend $F_0$ to a continuous map $F_0:X\times\C^{LN}\to Y$ with $F_0^{-1}(A)=f^{-1}(A)\times\C^{LN}\subset X'$.
Then by Theorem \ref{thm:okaprinciple}, there exists a holomorphic map $F_1:X\times\C^{LN}\to Y$ which agrees with $F_0$ to the second order along $X'=(X\times\{0\})\cup(f^{-1}(A)\times\C^{LN})$. 
Clearly, $F_1$ is then a holomorphic spray over $f$ dominating on $f^{-1}(U)$.
\end{proof}

Using Proposition \ref{prop:genell1}, let us prove the following refinement of Theorem \ref{thm:ell1}.
In the case of $U=\emptyset$, the equivalence $(1)\iff(2)$ is nothing but Theorem \ref{thm:ell1}.
As an application of Proposition \ref{prop:genell1} and this refinement, we shall prove Theorem \ref{thm:localization}.

\begin{cor}\label{cor:relell1}
For a complex manifold $Y$ with a Zariski open Oka subset $U\subset Y$, the following are equivalent:
\begin{enumerate}
\item $Y$ is Oka.
\item For any Stein manifold $X$ and any $f\in\cO(X,Y)$, there exists a holomorphic spray $X\times\C^N\to Y$ over $f$ dominating on $f^{-1}(Y\setminus U)$.
\item For any bounded convex domain $\Omega$ and any $f\in\cO(\Omega,Y)$, there exists a holomorphic spray $\Omega\times\C^N\to Y$ over $f$ dominating on $f^{-1}(Y\setminus U)$.
\end{enumerate}
\end{cor}

\begin{proof}
It suffices to prove the implication $(3)\implies(1)$.
Let $\Omega$ be a bounded convex domain and $f\in\cO(\Omega,Y)$.
By assumption, there exists a holomorphic spray $F:\Omega\times\C^N\to Y$ over $f$ dominating on $f^{-1}(Y\setminus U)$.
Since $\Omega\times\C^N$ is Stein, there exists a holomorphic spray $\widetilde F:(\Omega\times\C^N)\times\C^L\to Y$ over $F$ dominating on $F^{-1}(U)\supset f^{-1}(U)\times\{0\}$ by Proposition \ref{prop:genell1}.
This is also a holomorphic spray $\widetilde F:\Omega\times\C^{N+L}\to Y$ over $f$ dominating on $\Omega=f^{-1}(Y\setminus U)\cup f^{-1}(U)$.
Therefore $Y$ is an Oka manifold by Theorem \ref{thm:cell1}.
\end{proof}

\begin{proof}[Proof of Theorem \ref{thm:localization}]
Note that a complex manifold $Y$ is Oka if each compact subset of $Y$ has an open Oka neighborhood in $Y$ (see Definition \ref{defn:oka}).
Therefore, it suffices to prove the theorem in the case that $Y$ is a union of two Zariski open Oka subsets $U,V\subset Y$.
Take arbitrary holomorphic map $f:X\to Y$ from a Stein manifold $X$.
By Proposition \ref{prop:genell1}, there exists a holomorphic spray $X\times\C^N\to Y$ over $f$ dominating on $f^{-1}(V)\supset f^{-1}(Y\setminus U).$
Then $Y$ is Oka by Corollary \ref{cor:relell1}.
\end{proof}

\section{Applications}\label{sec:applications}

\subsection{Other new characterizations of Oka manifolds}

Here, we present some other new characterizations of Oka manifolds which follow from the results and the methods in the previous sections.

The condition (2) in the following can be seen as relative weak subellipticity.

\begin{cor}\label{cor:weaksubell1}
For a complex manifold $Y$, the following are equivalent:
\begin{enumerate}
\item $Y$ is Oka.
\item For any Stein manifold $X$, any $f\in\cO(X,Y)$ and any $x_0\in X$, there exist finitely many holomorphic sprays $F_j:X\times\C^{N_j}\to Y,\ j=1,\ldots,n$ over $f$ such that $\sum_{j=1}^n\partial_t|_{t=0}F_j(x_0,t)(\T_0\C^{N_j})=\T_{f(x_0)}Y$.
\end{enumerate}
\end{cor}

\begin{proof}
It suffices to prove the implication $(2)\implies(1)$.
Let $K\subset\C^n$ be a compact convex set and $f\in\cO(K,Y)$ be a holomorphic map.
Then there exists a convex domain $\Omega\supset K$ such that $f\in\cO(\Omega,Y)$.
Choose arbitrary $x_0\in K$.
By assumption, there exists a holomorphic spray $F_1:\Omega\times\C^{N_1}\to Y$ over $f$ such that $\partial_t|_{t=0}F_1(x_0,t)(\T_0\C^{N_1})\neq 0$.
By assumption again, there exists a holomorphic spray $F_2:(\Omega\times\C^{N_1})\times\C^{N_2}\to Y$ over $F_1$ such that $\partial_t|_{t=0}F_2((x_0,0),t)(\T_0\C^{N_2})\not\subset\partial_t|_{t=0}F_1(x_0,t)(\T_0\C^{N_1})$.
After repeating this process finitely many times, we obtain a holomorphic spray $F_k:\Omega\times\C^{N_1+\cdots+N_k}\to Y$ over $f$ such that 
\begin{align*}
\partial_t|_{t=0}F_k(x_0,t)(\T_0\C^{N_1+\cdots+N_k})=\sum_{j=1}^k\partial_t|_{t=0}F_j((x_0,0),t)(\T_0\C^{N_j})=\T_{f(x_0)}Y.
\end{align*}
Therefore, $F_k$ is dominating on a Zariski open neighborhood $U\subset\Omega$ of $x_0$.
If we repeat the above process for $F_k$ and a point in $K\setminus U$, after finitely many times, we may obtain a holomorphic spray $\Omega\times\C^N\to Y$ over $f$ dominating on $K$.
Shrinking $\Omega\supset K$ if necessary, it defines a dominating spray $\Omega\times\C^N\to Y$ over $f$ and hence $Y$ is an Oka manifold by Theorem \ref{thm:cell1}.
\end{proof}

In the following, if we consider the space $\Gamma(X,f^*\T Y)$ of holomorphic vector fields along $f$ as the tangent space $\T_f\cO(X,Y)$ to $\cO(X,Y)$ at $f$, the condition (2) means that we can draw an entire curve in any direction at each point of $\cO(X,Y)$.

\begin{cor}
For a complex manifold $Y$, the following are equivalent:
\begin{enumerate}
\item $Y$ is Oka.
\item For any Stein manifold $X$, any $f\in\cO(X,Y)$ and any $V\in\Gamma(X,f^* \T Y)$, there exists a holomorphic spray $F:X\times\C\to Y$ over $f$ which satisfies $(\partial F(\cdot,t)/\partial t)|_{t=0}=V$.
\end{enumerate}
\end{cor}

\begin{proof}
We first prove the implication $(1)\implies(2)$.
Consider $V\in\Gamma(X,f^*\T Y)$ as a holomorphic vector field on the graph $\Gamma_f\subset X\times Y$ tangent to the fibers of $X\times Y\to X$.
Since $\Gamma_f$ is Stein, we may extend it to an open neighborhood of $\Gamma_f$.
Using its local flow, we may construct a holomorphic map $F_{0}:U\to Y$ from an open neighborhood $U\subset X\times\C$ of $X\times\{0\}$ such that $F_{0}(\cdot,0)=f$ and $(\partial F_{0}(\cdot,t)/\partial t)|_{t=0}=V$ (see the proof of Proposition \ref{prop:genell1}).
Shrinking $U\supset X\times\{0\}$ if necessary, we can extend $F_{0}$ to a continuous map $F_{0}:X\times\C\to Y$.
Then by Forstneri\v{c}'s Oka principle (cf. \cite[Theorem 5.4.4]{ForSMHM}), we may obtain a holomorphic spray 
$F:X\times\C\to Y$ over $f$ which satisfies $(\partial F(\cdot,t)/\partial t)|_{t=0}=V$.

The converse implication follows from Corollary \ref{cor:weaksubell1} because $\Gamma(X,f^*\T Y)\to\T_{f(x_0)}Y,\ V\mapsto V(x_0)$ is surjective for any Stein manifold $X$ and any $x_{0}\in X$.
\end{proof}

Next, we shall consider strong $\C$-connectedness and prove the counterpart of Corollary \ref{cor:relell1} as an application of Theorem \ref{thm:okaprinciple}.
Recall the following characterization of Oka manifolds by strong $\C$-connectedness of mapping spaces.

\begin{thm}[{\cite[Corollary 3.1]{Kus}}]\label{thm:cconn}
For a complex manifold $Y$, the following are equivalent:
\begin{enumerate}
\item $Y$ is Oka.
\item For any Stein manifold $X$ and any pair $f_0,f_1\in\cO(X,Y)$ of homotopic holomorphic maps, there exists a holomorphic map $F:X\times\C\to Y$ such that $F(\cdot,0)=f_0$ and $F(\cdot,1)=f_1$.
\item For any bounded convex domain $\Omega$ and any pair $f_0,f_1\in\cO(\Omega,Y)$ of holomorphic maps, there exists a holomorphic map $F:\Omega\times\C\to Y$ such that $F(\cdot,0)=f_0$ and $F(\cdot,1)=f_1$.
\end{enumerate}
\end{thm}

The following is the counterpart of Corollary \ref{cor:relell1}.

\begin{thm}\label{thm:relcconn}
For a complex manifold $Y$ with a nonempty Zariski open Oka subset $U\subset Y$, the following are equivalent:
\begin{enumerate}
    \item $Y$ is Oka.
    \item For any Stein manifold $X$ and any $f\in\cO(X,Y)$ which is homotopic to some continuous map $X\to U$, there exists a holomorphic spray $F:X\times\C\to Y$ over $f$ with $F(\cdot,1)\in\cO(X,U)$.
    \item For any bounded convex domain $\Omega$ and any $f\in\cO(\Omega,Y)$, there exists a holomorphic spray $F:\Omega\times\C\to Y$ over $f$ with $F(\cdot,1)\in\cO(\Omega,U)$.
\end{enumerate}
\end{thm}

\begin{proof}
Since $U$ is Oka, any continuous map $X\to U$ from a Stein manifold $X$ is homotopic to a holomorphic map $X\to U$ by Forstneri\v{c}'s Oka principle (cf. \cite[Theorem 5.4.4]{ForSMHM}).
Thus the implication $(1)\implies(2)$ follows from Theorem \ref{thm:cconn}, and the implication $(2)\implies(3)$ is obvious.
Let us prove the remaining implication $(3)\implies(1)$.
For the proof, we shall use the characterization of Oka manifolds by strong $\C$-connectedness of pairs of nonempty open subsets in mapping spaces \cite[Theorem 3.2]{Kus}.
Let $\Omega$ be a bounded convex domain and $\cU_0,\cU_1\subset\cO(\Omega,Y)$ be a pair of nonempty open subsets.
By assumption, for each $j=0,1$, there exists a holomorphic map $F_j:\Omega\times\C\to Y$ such that $F_j(\cdot,0)\in\cU_j$ and $F_j(\cdot,1)\in\cO(\Omega,U)$.
Let $\pr_{2}:\Omega\times\C\to\C$ denote the second projection and set $A=Y\setminus U$.
Composing with a shrinking map $\Omega\to\Omega$ if necessary, we may assume that the closure of $\pr_{2}(F_j^{-1}(A))$ does not contain $1\in\C$ for each $j=0,1$.
Consider $G_j:\Omega\times\C\to Y,\ (z,t)\mapsto F_j(z,\exp(\pi it)+1)$.
Then $G_j(\cdot,1)\in\cU_j$ and there exists a large number $C_j>0$ such that $\pr_{2}(G_j^{-1}(A))\subset\{t\in\C:\Im t<C_j\}$ for each $j=0,1$, where $\Im t$ is the imaginary part of $t$.
Using these maps, we may construct a continuous map $H:\Omega\times\C\to Y$ which is holomorphic on an open neighborhood of $H^{-1}(A)\cup (\Omega\times\{0,1\})$, such that $H(\cdot,j)\in\cU_j,\ j=0,1$.
By Theorem \ref{thm:okaprinciple}, we obtain a holomorphic map $\widetilde H:\Omega\times\C\to Y$ such that $\widetilde H(\cdot,j)\in\cU_j,\ j=0,1$.
This implies that $Y$ is an Oka manifold by \cite[Theorem 3.2]{Kus}.
\end{proof}

\begin{rem}
Let $Y$ be a complex manifold and $A\subsetneq Y$ be a proper closed complex submanifold.
Assume that the complement $Y\setminus A$ is Oka.
Then for any $f\in\cO(\C^n,Y)$ there exists a holomorphic spray $F:\C^n\times\C\to Y$ over $f$ with $F(\cdot,1)\in\cO(\C^n,Y\setminus A)$.
Indeed, if $f^{-1}(A)\neq\C^n$ then $F(z,t)=f((1-t)z+tz_0)$ for $z_0\in\C^n\setminus f^{-1}(A)$ has the desired property.
In the case of $f^{-1}(A)=\C^n$, using the technique in the proof of \cite[Theorem 2]{ForLar}, we may construct a holomorphic spray $F:\C^n\times\C\to Y$ over $f$ such that $F(\cdot,t)\in\cO(\C^n,Y\setminus A)$ for all $t\in\C^*$. 
At present, we do not know whether this condition implies that $Y$ is Oka.
\end{rem}

\subsection{Affirmative answers to Gromov's conjectures}

As we mentioned in the introduction, Theorem \ref{thm:ell1} also gives affirmative answers to Gromov's conjectures in \cite[\S 1.4.E$''$]{Gro}.
These conjectures are the same thing as Exercises (d), (e) and (e$'$) in \cite[p.\,72]{GroPDR}.
We omit Exercise (e) because it is a special case of Exercise (e$'$).

\begin{conj}[{Gromov \cite[\S 1.4.E$''$]{Gro}}]\label{conj:gromov}
For a (connected) complex manifold $Y$ satisfying Condition $\Ell_1$, the following hold:
\begin{enumerate}
\item (Exercise (d) in \cite[p.\,72]{GroPDR}) Let $\Omega$ be a Runge domain in a complex manifold $X$.
Let $f_0\in\cO(X,Y)$ and $f_1\in\cO(\Omega,Y)$ be holomorphic maps.
Assume that $f_1$ is homotopic to $f_0|_\Omega$ through holomorphic maps.
Then $f_1$ can be approximated on $\Omega$ by holomorphic maps $X\to Y$.
\item (Exercise (e$'$) in \cite[p.\,72]{GroPDR}) For any Stein manifold $X$, any $f_0\in\cO(X,Y)$, any (closed) discrete subset $D\subset X$ and any map $\varphi:D\to Y$, there exists $f_1\in\cO(X,Y)$ which is homotopic to $f_0$ and satisfies $f_1|_D=\varphi$.
\end{enumerate}
\end{conj}

Under the natural Steinness assumption, we may solve the above exercises.

\begin{cor}
Conjecture \ref{conj:gromov} is true for Stein manifolds $X$ and $\Omega$.
\end{cor}

\begin{proof}
We may assume that a complex manifold $Y$ is Oka by Theorem \ref{thm:ell1}.
Then both are well-known properties but we give their proofs for reader's convenience.
\\
(1) Since $X$ and $\Omega$ is Stein, the Runge domain $\Omega\subset X$ is exhausted by compact $\cO(X)$-convex subsets (cf. \cite[Theorem 2.3.3]{ForSMHM}).
Take a sufficiently large compact $\cO(X)$-convex subset $K\subset\Omega$.
By using a homotopy between $f_1$ and $f_0|_\Omega$, we may construct a continuous map $X\to Y$ which agrees with $f_1$ on an open neighborhood of $K\subset X$.
Then Forstneri\v{c}'s Oka principle (cf. \cite[Theorem 5.4.4]{ForSMHM}) implies that there exists a holomorphic map $X\to Y$ which approximates $f_1$ on $K$.
\\
(2) By deforming $f_0$ near each point of $D$, we may construct a continuous map $f_1':X\to Y$ which is homotopic to $f_0$ and satisfies $f_1'|_D=\varphi$.
Then Forstneri\v{c}'s Oka principle implies that there exists a holomorphic map $f_1:X\to Y$ which is homotopic to $f_1'$ and satisfies $f_1|_D=\varphi$.
Note that $f_1$ is also homotopic to $f_0$.
\end{proof}

Recall that, in \cite{GroPDR} and \cite{Gro}, Gromov also introduced the conditions $\Ell_2$ and $\Ell_\infty$ for complex manifolds.
Despite their names, $\Ell_2$ and $\Ell_\infty$ are not ellipticity conditions but Oka properties.
In modern terms, Gromov's conditions $\Ell_2$ and $\Ell_\infty$ are BOPJI and POPAJI for polyhedral parameters, respectively (for the definitions of these terms, see \cite[\S 5.15]{ForSMHM}).
As he mentioned in \cite[p.\,73]{GroPDR}, it has also been unknown whether $\Ell_1$ implies $\Ell_2$.
Now we know that $\Ell_2$, $\Ell_\infty$ and being Oka are equivalent (cf. \cite[Proposition 5.15.1]{ForSMHM}).
Combining this fact with Theorem \ref{thm:ell1}, we may obtain the following equivalences between Gromov's $\Ell$ conditions.

\begin{cor}
Gromov's conditions $\Ell_1$, $\Ell_2$ and $\Ell_\infty$ are equivalent.
\end{cor}

\subsection{New examples of Oka manifolds}

We give here some new examples of Oka manifolds by using the localization principle (Theorem \ref{thm:localization}).

Since $\P^n$ is locally isomorphic to $\C^n$ which is flexible, the complement of an algebraic subvariety of codimension at least two in $\P^n$ is Oka (cf. \cite[Theorem 0.1]{FleKalZai}).
In contrast to this case, the complement of a hypersurface in $\P^n$ is only rarely Oka (recall the logarithmic Kobayashi conjecture, see \cite{BroDen} for a recent result).
Hanysz \cite{Han} determined when complements of unions of hyperplanes are Oka.
It is a well-known problem whether the complement of every smooth cubic curve in $\P^2$ is Oka.
Related to these, we prove the Oka property for the following complements.

\begin{cor}
{\rm (1)} The complement of any quadric hypersurface in $\P^n$ is Oka.
\\
{\rm (2)} The complement of any singular irreducible cubic curve in $\P^2$ is Oka.
\end{cor}

\begin{proof}
(1) After a change of coordinates, we may assume that a given hypersurface is $Q=\{z_0^2+\cdots+z_k^2=0\}\subset\P^n,\ 0\leq k\leq n$.
If $k=0$, the complement $\P^n\setminus Q\cong\C^n$ is clearly Oka.
Let us consider the case of $k\geq 1$.
Take arbitrary point $a=[a_0,\ldots,a_n]\in\P^n\setminus Q$.
Since $a_j\neq 0$ for some $0\leq j\leq k$, we may assume that $a_0\neq 0$ and $a_0+ia_k\neq 0$ by a change of coordinates.
If we set $w_0=z_0+iz_k$ and $w_k=z_0-iz_k$, the equation of $Q$ becomes $z_0^2+\cdots+z_k^2=w_0w_k+z_1^2+\cdots+z_{k-1}^2$.
Thus, the point $a\in\P^n\setminus Q$ has a Zariski open neighborhood which is isomorphic to the complement $(\C^{n-1}\times\C)\setminus\Gamma_f$ of the graph of $f(x_1,\ldots,x_{n-1})=-(x_1^2+\cdots+x_{k-1}^2)$.
Since the complement $(\C^{n-1}\times\C)\setminus\Gamma_f\cong\C^{n-1}\times\C^*$ is Oka, the complement $\P^n\setminus Q$ is also Oka by the localization principle. 
\\
(2) After a change of coordinates, a given curve $C$ can be assumed to be $\{y^2z=x^3\}$ or $\{y^2z=x^3+x^2z\}$.
In both cases, the complement $\P^2\setminus C$ is Zariski locally isomorphic to the complement $(\C\times\C)\setminus\Gamma_f$ of the graph of a rational function $f$ in one variable.
By the result of Hanysz \cite[Theorem 4.6]{Han} the complement $(\C\times\C)\setminus\Gamma_f$ is Oka, hence the conclusion follows from the localization principle.
\end{proof}

Next, we shall prove the Oka property for some blowups of $\C^n$.
As we shall see in the appendix, the blowup of $\C^n$ is in general not Oka (Example \ref{ex:blowup}).
Let us first recall the following result of Forstneri\v{c}.

\begin{prop}[{\cite[Proposition 6.4.12]{ForSMHM}}]\label{prop:tame}
Let $n>1$ and $D\subset\C^n$ be a discrete subset.
Assume that $D\subset\C^n$ is tame, i.e. there exists an automorphism $\varphi\in\Aut\C^n$ such that $\varphi(D)=\N\times\{0\}\subset\C^n$.
Then the blowup $\Bl_D\C^n$ is Oka.
\end{prop}

There is also the notion of tameness for higher dimensional subvarieties (cf. \cite[Definition 4.11.3]{ForSMHM}).
It seems that the above also holds for any tame submanifold $A\subset\C^n$ of codimension at least two.
Here we give the following generalization of Proposition \ref{prop:tame} under the stronger tameness assumption.
Note that in the case of $\dim A=0$ it is nothing but Proposition \ref{prop:tame} (cf. \cite[Corollary 4.6.3\,(a)]{ForSMHM}).

\begin{cor}
Let $A\subset\C^n$ be a closed complex submanifold of pure dimension $k\leq n-2$.
Assume that there exists an automorphism $\varphi\in\Aut\C^n$ such that $\varphi(A)\subset\C^{k+1}\times\{0\}\subset\C^n$.
Then the blowup $\Bl_A\C^n$ is Oka.
\end{cor}

\begin{proof}
We may assume that $A$ is a smooth hypersurface in $\C^{k+1}\times\{0\}\subset\C^n$.
Since the second Cousin problem on $\C^n$ is solvable (cf. \cite[\S 5.2]{ForSMHM}), there exists a holomorphic function $f\in\cO(\C^{k+1})$ such that $A=\{z\in\C^{k+1}:f(z)=0\}$ and $\d f\neq 0$ on $A$.
Then the blowup $\Bl_A\C^n$ can be described as follows:
\begin{align*}
\Bl_A\C^n=\{(z,z',w)\in\C^{k+1}\times\C^{n-k-1}\times\P^{n-k-1}:w=[f(z),z'_1,\ldots,z'_{n-k-1}]\}.
\end{align*}
Note that this is covered by the Zariski open subset
\begin{align*}
\{(z,z',w)\in\C^{k+1}\times\C^{n-k-1}\times\C^{n-k-1}&:z'=f(z)w\}\cong\C^n
\end{align*}
and Zariski open subsets of the same form
\begin{align*}
&\{(z,z',w)\in\C^{k+1}\times\C^{n-k-1}\times\C^{n-k-1}:f(z)=z'_1w_1,\ (z'_j)_{j=2}^{n-k-1}=z'_1(w_j)_{j=2}^{n-k-1}\} \\
\cong\,&\{(z,z'_1,w_1)\in\C^{k+1}\times\C\times\C:f(z)=z'_1w_1\}\times\C^{n-k-2}.
\end{align*}
Since the hypersurface $\{(z,z'_1,w_1)\in\C^{k+1}\times\C\times\C:f(z)=z'_1w_1\}$ is Oka by the result of Kaliman and Kutzschebauch \cite[Theorem 2]{KalKut}, the blowup $\Bl_A\C^n$ is also Oka by the localization principle.
\end{proof}

\appendix
\section{Non-Oka Blowups of $\C^n$}

In this appendix, we shall construct for each $n>1$ a discrete subset $D\subset\C^n$ such that the blowup $\Bl_D\C^n$ of $\C^n$ along $D$ is not Oka (Example \ref{ex:blowup}).
Recall that the blowup of $\C^n$ along any tame discrete subset is Oka (Proposition \ref{prop:tame}).

The construction is essentially due to Rosay and Rudin; more precisely, it is based on their construction of a discrete subset in $\C^n$ which is unavoidable by nondegenerate maps \cite[Theorem 4.5]{RosRud}.
Since our construction is almost the same, we omit the details and refer the reader to \cite{RosRud}.
The only additional fact we need is that for a complex manifold $Y$ and a closed complex submanifold $A\subset Y$ of codimension at least two, the blowup $\pi:\Bl_AY\to Y$ is not a submersion at any point of $\pi^{-1}(A)\subset\Bl_AY$.

In the construction below, we always assume that $n>1$ and a discrete subset is closed.
The following lemma can be proved in the same way as \cite[Lemma 4.3]{RosRud}.
Here, $JF(z)$ is the Jacobian determinant of $F$ at $z\in\C^n$, and $\B^n=\{z\in\C^n:|z|<1\}$ is the unit ball in $\C^n$.

\begin{lem}
Given $0<a_1<a_2$, $0<r_1<r_2$ and $c>0$, let $\Gamma$ denote the set of all $F\in\cO(a_2\B^n,r_2\B^n)$ such that $|F(0)|\leq r_1/2$ and $\sup_{|z|\leq a_1}|JF(z)|\geq c$.
Then there exists a finite set $E=E(a_1,a_2,r_1,r_2,c)\subset\partial(r_1\B^n)$ with the following property: If $F\in\Gamma$ factors through the blowup $\Bl_E(r_2\B^n)\to r_2\B^n$, then $F(a_1\B^n)\subset r_1\B^n$ holds.
\end{lem}

For $k\in\N$, choose $k/2=a_1<a_2<\cdots<3k/4$ and $k=r_1<r_2<\cdots$ with $\lim_{j\to\infty}r_j=\infty$.
Using the notation of the above lemma, we define $D_k=\bigcup_{j\in\N}E(a_j,a_{j+1},r_j,r_{j+1},1/k)\subset\C^n$.
It is a discrete subset and has the following property.
For the details, see the proof of \cite[Lemma 4.4]{RosRud}.

\begin{lem}\label{lem:Dk}
For each $k\in\N$, there exists a discrete subset $D_k\subset\C^n\setminus k\B^n$ with the following property:
If $F\in\cO(k\B^n,\C^n)$ satisfies
\begin{enumerate}
\item $|F(0)|\leq k/2$,
\item $\sup_{|z|\leq k/2}|JF(z)|\geq 1/k$, and
\item $F$ factors through the blowup $\Bl_{D_k}\C^n\to\C^n$,
\end{enumerate}
then the inclusion $F((k/2)\B^n)\subset k\B^n$ holds.
\end{lem}

Recall that an $n$-dimensional complex manifold $Y$ is said to be {\em Brody volume hyperbolic} if there is no holomorphic map $\C^n\to Y$ which is locally biholomorphic at some point (such a map is said to be {\em nondegenerate}).
Clearly, no Brody volume hyperbolic manifold is Oka.
By using the above lemmas, we may construct the following example of a discrete subset with the desired property.

\begin{ex}\label{ex:blowup}
For each $n>1$, let us consider the discrete subset $D=\bigcup_{k\in\N}D_k\subset\C^n$ by using $D_k$ in Lemma \ref{lem:Dk}.
Then the blowup $\Bl_D\C^n$ is Brody volume hyperbolic.
It can be shown as follows.
Let $\pi:\Bl_D\C^n\to\C^n$ denote the blowup of $\C^n$ along $D$.
Assume that there exists a nondegenerate map $f:\C^n\to\Bl_D\C^n$.
Then the composition $F=\pi\circ f:\C^n\to\C^n$ is also nondegenerate and $F$ factors through each of the blowups $\Bl_{D_k}\C^n\to\C^n,\ k\in\N$.
By construction, $F((k/2)\B^n)\subset k\B^n$ for all sufficiently large $k\gg 0$.
This growth condition and nondegeneracy imply that $F$ is an affine isomorphism.
This contradicts the definition of $F=\pi\circ f$.
\end{ex}

\section*{Acknowledgement}

I would like to thank my supervisor Katsutoshi Yamanoi for many helpful comments.
In particular, I learned from him how to construct Example \ref{ex:blowup}.
I also thank Franc Forstneri\v{c} for useful remarks and suggestions.
This work was supported by JSPS KAKENHI Grant Number JP18J20418.

\end{document}